\documentclass[12pt]{amsart}
\usepackage[T1]{fontenc}
\usepackage{lmodern}
\usepackage{amsmath,amssymb,amsthm,mathtools}
\usepackage{microtype}
\usepackage[hidelinks]{hyperref}
\usepackage[margin=1.15in]{geometry}
\usepackage{enumitem}

\newtheorem{theorem}{Theorem}[section]
\newtheorem{proposition}[theorem]{Proposition}
\newtheorem{lemma}[theorem]{Lemma}
\newtheorem{corollary}[theorem]{Corollary}
\theoremstyle{definition}
\newtheorem{definition}[theorem]{Definition}
\newtheorem{example}[theorem]{Example}
\theoremstyle{remark}
\newtheorem{remark}[theorem]{Remark}

\newcommand{\Map}{\operatorname{Map}}
\newcommand{\Sym}{\operatorname{Sym}}
\newcommand{\mdim}{\operatorname{mdim}}

\newcommand{\diam}{\operatorname{diam}}
\newcommand{\cH}{\mathsf H}
\newcommand{\N}{\mathbb N}

\title[Sofic $p$-metric mean dimension]{Equivalence of Sofic $p$-Metric Mean Dimensions and a Tame-Metric Variational Formula}
\author{Xianqiang Li and Zhuowei Liu}
\date{\today}
\subjclass[2020]{37B02, 54E45}
\address[X. Li]{Shanghai center for Mathematial Science, Fudan University, Shanghai,  200000, P.R. China}
\email{26110840010@m.fudan.edu.cn}
\address[Z. Liu]{School of Mathematics (Zhuhai), Sun Yat-sen University,
	Zhuhai, Guangdong, 519000, P.R. China}
\email{liuzhw55@mail2.sysu.edu.cn}

\keywords{Sofic group; metric mean dimension}
\thanks{Zhuowei Liu is the corresponding author.}

\begin{document}

\begin{abstract}
Let $\Gamma$ be a countable discrete sofic group acting by homeomorphisms on a compact metrizable
space $X$ and $\Sigma$ a sofic approximation of $\Gamma$. We prove that for every $1\leq p<\infty$, the sofic $p$-metric mean dimension is equivalent to
the sofic metric mean dimension, i.e there exists a common value
$ D_\Sigma(X,\Gamma)\in\{-\infty\}\cup[0,+\infty]$
such that,
$$D_\Sigma(X,\Gamma)=\mdim_{\Sigma,\mathrm M,p}(X,\Gamma)
 =\mdim_{\Sigma,\mathrm M,\infty}(X,\Gamma),$$
 which answers a question of Hayes in \cite[Question 3]{Hayes}.
 
Moreover, a tame-metric variational formula is established. That is  for every $1\leq q\leq\infty$,
$$
 D_\Sigma(X,\Gamma)
 =\inf_{\rho\in\mathcal T(X)}
 \underline{\mdim}_{\Sigma,q}(X,\rho),
$$
where $\mathcal T(X)$ is the set of all compatible metrics on $X$ having tame growth of covering numbers.
\end{abstract}

\maketitle

\section{Introduction}

Mean dimension, introduced by Gromov~\cite{MG} and systematically developed by Lindenstrauss and Weiss~\cite{LW}, is a topological invariant that quantifies the “number of parameters per unit action” in infinite-dimensional dynamical systems. Building on this, Lindenstrauss and Weiss introduced a metric analogue, called metric mean dimension. It is defined by taking the infimum, over all metrics on the space $X$, of the dynamical packing numbers associated with the system.

For action of countable discrete sofic groups, Li~\cite{Li} extended metric mean dimension via sofic approximations, leading to the notion of sofic metric mean dimension $\mdim_{\Sigma,\mathrm M,\infty}(X,\Gamma)$, which depends on a fixed sofic approximation $\Sigma$.
Subsequently, Hayes \cite{Hayes} introduced $\ell^p$ analogues of sofic metric mean dimension and asked whether, for every $1\leq p<\infty$,
\[
 \mdim_{\Sigma,\mathrm M,p}(X,\Gamma)
 =\mdim_{\Sigma,\mathrm M,\infty}(X,\Gamma).
\]
See \cite[Question 3]{Hayes}. The purpose of this paper is to prove the above equality.

There is an important fixed-metric precursor. Hayes proved that if a dynamically generating continuous pseudometric $\rho$ has finite upper covering exponent,
\[
 \limsup_{\varepsilon\downarrow0}
 \frac{\log S_\varepsilon(X,\rho)}{\log(1/\varepsilon)}<+\infty,
\]
then the fixed-metric $\ell^1$ and $\ell^\infty$ quantities coincide \cite[Proposition 2.20]{Hayes}. Hayes called this a finite packing-dimension hypothesis; in standard metric-dimension terminology the displayed condition is an upper box/Minkowski-type covering condition. 

In this paper, we extend that mechanism from finite packing dimension to the larger class of metrics with tame growth of covering numbers (which introduced by Lindenstrauss and Tsukamoto \cite{LT}). Firstly, we show that for a  compatible metric $\rho$ on $X$ with tame growth of covering numbers and every $1\leq p<\infty$, the sofic 
$p$-metric mean dimensions with respect to a fixed-metric is equal to sofic metric mean dimension with respect to a fixed-metric (see Proposition~\ref{prop:tame-equality}).

The decisive new step for \cite[Question 3]{Hayes} is then to combine this extension with a tame metric minorant and a monotonicity theorem under decreasing the phase-space metric. Let $\mathcal T(X)$ denote the set of all compatible metrics on $X$ having tame growth of covering numbers (see Definition \ref{def tame gowth}). Precisely, we show the following tame-metric variational formula.
\begin{theorem}[Tame-metric variational formula]\label{thm:strengthened}
Let $\Gamma$ be a countable discrete sofic group acting by homeomorphisms on a compact metrizable
space $X$ and $\Sigma$ a sofic approximation of $\Gamma$. There exists a common value
\[
 D_\Sigma(X,\Gamma)\in\{-\infty\}\cup[0,+\infty]
\]
such that, for every $1\leq p<\infty$,
\[
 D_\Sigma(X,\Gamma)
 =\mdim_{\Sigma,\mathrm M,p}(X,\Gamma)
 =\mdim_{\Sigma,\mathrm M,\infty}(X,\Gamma).
 \tag{1.1}
\]
Moreover, for every $1\leq q\leq\infty$,
\[ 
 D_\Sigma(X,\Gamma)
 =\inf_{\rho\in\mathcal T(X)}
 \underline{\mdim}_{\Sigma,q}(X,\rho).
 \tag{1.2}
\]
Thus the same tame class of compatible metrics computes all the $\ell^q$ versions simultaneously.
\end{theorem}

Moreover, we show that the the infimum in Theorem~\ref{thm:strengthened} can take over all Hilbert-cube metric, which is the Hilbert-cube metrics variational formula (see Theorem~\ref{thm:hilbert-formula}).

This paper is organized as follows. 
In Section~\ref{sec:prelim}, we recall the definitions and give some useful lemmas. In Section~\ref{sec:l1linfty}, for $1\leq p < \infty$, given  a compatible metric which has tame growth of covering numbers, we prove the equivalence between the lower (upper) sofic $p$-metric mean dimension and the lower (upper) sofic metric mean dimension with respect to this metric. In Section~\ref{sec:main}, we prove the Theorem~\ref{thm:strengthened} and answer the question of Hayes. In Section~\ref{sec:hilbert}, we confine the variational formula to an explicit family of Hilbert-cube metrics and derive an improved variational formula for this case.

\emph{Acknowledgments.} We would like to thank Professor Siming Tu, who provided much guidance and assistance.

 \section{Definitions and notation}\label{sec:prelim}
\begin{definition}
For a pseudometric space $(X,\rho)$ and $\varepsilon>0$, a subset $E\subseteq X$ is called $(\rho,\varepsilon)$-spanning if for every $y\in X$ there exists $x\in E$ with $\rho(x,y)<\varepsilon$. We write
$\mathrm{S}_\varepsilon(X,\rho)$
for the least cardinality of an $(\rho,\varepsilon)$-spanning subset of $X$. We use the convention $\mathrm{S}_\varepsilon(\varnothing,\rho)=0$ and $\log 0=-\infty$.
\end{definition}

\begin{definition}
For $d \in \N$, we write $[d]$ for the set $\{1,\cdots,d\}$ and  \(\Sym(d)\) for the permutation group of $[d]$. $|\cdot|$ denotes the cardinality of a set. A countable group $\Gamma$ is called \emph{sofic} if there is a \emph{sofic approximation sequence} $\Sigma=\{\sigma_{i}:\Gamma\rightarrow \text{Sym}(d_{i})\}_{i=1}^{\infty}$ for $\Gamma$, namely the following three conditions are satisfied:
		\begin{itemize}
			\item [(1)]for any $s,t \in \Gamma$, one has 
			$$\lim_{i \rightarrow \infty}\frac{|\{v \in [d_{i}]:\sigma_{i}(s)\sigma_{i}(t)(v)=\sigma_{i}(st)(v)\}|}{d_{i}}=1;$$
			\item [(2)]for any distinct $s,t \in \Gamma$, one has $$\lim_{i \rightarrow \infty}\frac{|\{v \in [d_{i}]:\sigma_{i}(s)(v)=\sigma_{i}(t)(v)\}|}{d_{i}}=0;$$
			\item [(3)] $\lim_{i \rightarrow \infty}d_{i}=+\infty.$
		\end{itemize}    
\end{definition}

For a map $\sigma$ from $G$ to $\Sym(d)$ for some $d \in \N$, we write $\sigma(s)(a)$ as $\sigma_{s}(a)$, when there is no confusion.

Throughout this paper, $\Gamma$ is a countable discrete sofic group and $X$ is a compact metrizable space.
\begin{definition}
Let $\rho$ be a continuous pseudometric on $X$. For maps $\varphi,\psi:[d]\to X$ and $1\leq p<\infty$, set
\[
 \rho_p(\varphi,\psi)
 =\left(\frac1d\sum_{v=1}^d \rho(\varphi(v),\psi(v))^p\right)^{1/p},
\]
and
\[
 \rho_\infty(\varphi,\psi)=\max_{1\leq v\leq d}\rho(\varphi(v),\psi(v)).
\]
For a nonempty finite set $F\subseteq\Gamma$, $\delta>0$, and a map $\sigma:\Gamma\to\Sym(d)$, define
\[
 \Map(\rho,F,\delta,\sigma)
 =\left\{\varphi:[d]\to X:
 \rho_2(\varphi\circ\sigma_s,s\varphi)<\delta
 \text{ for every }s\in F\right\}.
\]
Here $(s\varphi)(v)=s\varphi(v)$.
\end{definition}

\begin{definition}
A continuous pseudometric $\rho$ is dynamically generating if for every distinct $x,y\in X$ there exists $s\in\Gamma$ such that $\rho(sx,sy)>0$.
\end{definition}
 Recall that every compatible metric is dynamically generating, since the identity element already separates distinct points.

\begin{definition}
Let $\Gamma$ be a countable discrete sofic group acting by homeomorphisms on a compact metrizable
space $X$, and let $\Sigma$ be a sofic approximation of $\Gamma$ and $\rho$ a dynamically generating continuous pseudometric on $X$.
For $1\leq p\leq\infty$ and $\epsilon>0$, we define 
\[
 h_{\Sigma,p}(\varepsilon,\rho)
 =\inf_{F\Subset\Gamma}\inf_{\delta>0}
 \limsup_{i\to\infty}\frac1{d_i}
 \log \mathrm{S}_\varepsilon\bigl(\Map(\rho,F,\delta,\sigma_i),\rho_p\bigr),
 \tag{2.1}
\]
where $F\Subset\Gamma$ means that $F$ is a nonempty and finite subset of $\Gamma$. If $\Map(\rho,F,\delta,\sigma_i)$ is
empty for all sufficiently large $i$, we set 
$ h_{\Sigma,p}(\varepsilon,\rho)=-\infty$. 
Define the \emph{lower sofic $p$-metric mean dimension with respect to  $\rho$}  as
\[
 \underline{\mdim}_{\Sigma,p}(X,\rho)
 =\liminf_{\varepsilon\downarrow0}
 \frac{h_{\Sigma,p}(\varepsilon,\rho)}{\log(1/\varepsilon)}.
 \tag{2.2}
\]
We also  define the \emph{upper sofic $p$-metric mean dimension with respect to  $\rho$} as 
\[
 \overline{\mdim}_{\Sigma,p}(X,\rho)
 =\limsup_{\varepsilon\downarrow0}
 \frac{h_{\Sigma,p}(\varepsilon,\rho)}{\log(1/\varepsilon)}.
 \tag{2.3}
\]
Finally, we define the \emph{ sofic $p$-metric mean dimension } by
\[
 \mdim_{\Sigma,\mathrm M,p}(X,\Gamma)
 =\inf_\rho \underline{\mdim}_{\Sigma,p}(X,\rho),
 \tag{2.4}
\]
where the infimum is taken over dynamically generating continuous pseudometrics. For $p=\infty$ this is the usual sofic metric mean dimension as defined in Li \cite{Li}.
\end{definition}
\begin{remark}\label{rem:compatible}
In the definition of sofic $p$-metric mean dimension,  the infimum can equivalently be taken over all compatible metrics $\rho$ on $X$. See \cite{Hayes,Li} for details.
\end{remark}

We need the following Lemmas.
\begin{lemma}\cite[Lemma 2.14]{Hayes}
\label{lem:hayes-input}
Let $\Gamma$ be a countable discrete sofic group acting by homeomorphisms on a compact metrizable
space $X$, and let $\Sigma$ be a sofic approximation of $\Gamma$.
 Let $\rho$ and $\rho'$ be dynamically generating continuous pseudometrics. For every finite $F\subseteq\Gamma$ and every $\delta>0$, there exist a finite $F'\subseteq\Gamma$ and $\delta'>0$ such that, for all sufficiently large $i$,
 \[
  \Map(\rho',F',\delta',\sigma_i)
  \subseteq \Map(\rho,F,\delta,\sigma_i).
  \tag{2.5}
 \]

\end{lemma}

\begin{lemma}[Metric independence of the empty-model alternative]
\label{lem:empty-independent}
Let $\rho$ and $\rho'$ be dynamically generating continuous pseudometrics on $X$. The following are equivalent:
\begin{enumerate}[label=\textup{(\roman*)}]
 \item there exist $F\Subset\Gamma$ and $\delta>0$ such that
 $\Map(\rho,F,\delta,\sigma_i)=\varnothing$ for all sufficiently large $i$;
 \item there exist $F'\Subset\Gamma$ and $\delta'>0$ such that
 $\Map(\rho',F',\delta',\sigma_i)=\varnothing$ for all sufficiently large $i$.
\end{enumerate}
\end{lemma}

\begin{proof}
It follows immediately from Lemma \ref{lem:hayes-input}.
\end{proof}

\begin{remark}[Global empty-model alternative]\label{rem:empty-models}
By Lemma~\ref{lem:empty-independent}, eventual emptiness is independent of the chosen dynamically generating continuous pseudometric. If it occurs, then
\[
 h_{\Sigma,p}(\varepsilon,\rho)=-\infty
\]
for every dynamically generating $\rho$, every $\varepsilon>0$, and every $1\leq p\leq\infty$; hence all lower and  upper sofic $p$-metric mean dimension with respect to  $\rho$ equal $-\infty$. We also have the sofic $p$-metric mean dimension equal $-\infty$.

Otherwise, for every dynamically generating $\rho$, every $F\Subset\Gamma$, and every $\delta>0$, the corresponding model space is nonempty for infinitely many $i$. Since every nonempty pseudometric space has covering number at least one, the relevant $\limsup$ is nonnegative. Consequently
\[
 h_{\Sigma,p}(\varepsilon,\rho)\geq0
\]
for every $\varepsilon>0$ and every $1\leq p\leq\infty$. In this second alternative all lower and  upper sofic $p$-metric mean dimension with respect to  $\rho$, and sofic $p$-metric mean dimension, belong to $[0,+\infty]$.
\end{remark}

\begin{lemma}\label{lem:scale-change}
Let $\varepsilon_0>0$, $c>0$ and $\alpha>0$. Let $a,b:(0,\varepsilon_0)\to[0,+\infty]$, and let $e:(0,\varepsilon_0)\to\mathbb R$ satisfy $e(t)\to0$ as $t\downarrow0$. Suppose that
\[
 a(ct^\alpha)\leq b(t)+e(t)
\]
for all sufficiently small $t>0$. Then
\[
 \liminf_{r\downarrow0}\frac{a(r)}{\log(1/r)}
 \leq \frac1\alpha
 \liminf_{t\downarrow0}\frac{b(t)}{\log(1/t)},
\]
\[
 \limsup_{r\downarrow0}\frac{a(r)}{\log(1/r)}
 \leq \frac1\alpha
 \limsup_{t\downarrow0}\frac{b(t)}{\log(1/t)}.
\]
\end{lemma}

\begin{proof}
Note that $r=ct^\alpha$ is a bijection between sufficiently small positive $t$ and sufficiently small positive $r$. Moreover, when $t\downarrow0$,
\[
 \frac{\log(1/t)}{\log(1/(ct^\alpha))}\longrightarrow\frac1\alpha
 \quad\text{and}\quad
 \frac{e(t)}{\log(1/(ct^\alpha))}\longrightarrow0,
\]
Since $ a(ct^\alpha)\leq b(t)+e(t)$ for all sufficiently small $t>0$,
    then
    \[
        \frac{a(ct^{\alpha})}{\log(1/(ct^{\alpha}))}\leq \frac{b(t)}{\log(1/t)} \frac{\log(1/t)}{\log(1/(ct^\alpha))}+ \frac{e(t)}{\log(1/(ct^\alpha))}.
    \]
As $t\downarrow0 $  iff $r\downarrow0$, then we can get the desire result.
\end{proof}

\section{The \texorpdfstring{$\ell^1$--$\ell^\infty$}{l1--linfinity} comparison for tame metrics}\label{sec:l1linfty}
This section absorbs and develops the ideas of Hayes \cite[Proposition 2.20]{Hayes}. That work locates sparse mismatching coordinates via a comparison between $\ell^1$ and $\ell^\infty$ coverings; the following lemma extends this to settings with arbitrary covering growth. The novelty here lies in analyzing the resulting error term under tame growth conditions and integrating it with the metric-minorant argument of Proposition \ref{prop:monotone} and Lemma \ref{lem:tame-minorant}.

\begin{proposition}\label{prop:monotone}
Let $\Gamma$ be a countable discrete sofic group acting by homeomorphisms on a compact metrizable
space $X$ and $\Sigma$ a sofic approximation of $\Gamma$.
Let $\rho_0$ and $\rho_1$ be dynamically generating continuous pseudometrics on $X$ satisfying
\[
 \rho_0\leq \rho_1.
\]
Then, for every $1\leq p\leq\infty$,
\[
 \underline{\mdim}_{\Sigma,p}(X,\rho_0)
 \leq \underline{\mdim}_{\Sigma,p}(X,\rho_1)
\]
and
\[
 \overline{\mdim}_{\Sigma,p}(X,\rho_0)
 \leq \overline{\mdim}_{\Sigma,p}(X,\rho_1).
\]
\end{proposition}

\begin{proof}
Fix $\varepsilon>0$, a finite set $F\subseteq\Gamma$, and $\delta>0$. By (2.5), there exist $F'\Subset\Gamma$ and $\delta'>0$ such that, for all sufficiently large $i$,
\[
 A_i:=\Map(\rho_0,F',\delta',\sigma_i)
 \subseteq
 B_i:=\Map(\rho_1,F,\delta,\sigma_i).
 \tag{3.1}
\]

Let $E_i\subseteq B_i$ be a $(\rho_{1,p},\varepsilon)$-spanning set  with the least cardinality $ \mathrm{S}_\varepsilon(B_i,\rho_{1,p})$. For every $z\in E_i$ whose open $\rho_{1,p}$-ball of radius $\varepsilon$ meets $A_i$, choose one point $a_z$ in that intersection. If $x\in A_i$, choose $z\in E_i$ with $\rho_{1,p}(x,z)<\varepsilon$. Then
\[
 \rho_{0,p}(x,a_z) \leq \rho_{1,p}(x,a_z)
 \leq \rho_{1,p}(x,z)+\rho_{1,p}(z,a_z)
 <2\varepsilon.
\]
Consequently,
\[
 \mathrm{S}_{2\varepsilon}(A_i,\rho_{0,p})
 \leq \mathrm{S}_\varepsilon(B_i,\rho_{1,p}),
 \tag{3.2}
\]
 which implies that
\[
 h_{\Sigma,p}(2\varepsilon,\rho_0) \leq \limsup_{i\to\infty}\frac1{d_i}
 \log \mathrm{S}_{2\varepsilon}(A_i,\rho_{0,p})
 \leq
 \limsup_{i\to\infty}\frac1{d_i}
 \log \mathrm{S}_\varepsilon(B_i,\rho_{1,p}).
\]
Taking the infimum over the originally arbitrary $F$ and $\delta$ gives
\[
 h_{\Sigma,p}(2\varepsilon,\rho_0)
 \leq h_{\Sigma,p}(\varepsilon,\rho_1).
 \tag{3.3}
\]
This implies the desired result.
If the global empty-model alternative of Remark~\ref{rem:empty-models} holds, the conclusion is immediate.
\end{proof}
\begin{definition}\label{def tame gowth}\cite[Definition 3.8]{LT}
A compact metric space $(X,\rho)$ has \emph{tame growth of covering numbers}, if for every $a>0$,
\[
 \lim_{\varepsilon\downarrow0}
 \varepsilon^a\log \mathrm{S}_\varepsilon(X,\rho)=0.
 \tag{3.4}
\]
\end{definition}

The following lemma is obtained from the standard Hilbert-cube construction used by Lindenstrauss and Tsukamoto \cite[Lemma 3.10]{LT}. We include the complete proof because the order relation $\rho'\leq\rho$, which is essential for our  argument, is not merely a topological equivalence statement.

\begin{lemma}\label{lem:tame-minorant}
Let $(X,\rho)$ be a compact metric space. There exists a compatible metric $\rho'$ on $X$ such that
\[
 \rho'\leq\rho
\]
and $(X,\rho')$ has tame growth of covering numbers.
\end{lemma}

\begin{proof}
Set
\[
 M=\max\{1,\diam(X,\rho)\},
 \qquad \widetilde\rho=\rho/M.
\]
Then $\diam(X,\widetilde\rho)\leq1$ and
$\widetilde\rho\leq\rho$. Choose a dense sequence $(x_n)_{n\geq1}$ in $X$ and define
\[
 \Phi:X\longrightarrow [0,1]^{\N},
 \qquad
 \Phi(x)=(\widetilde\rho(x,x_n))_{n\geq1}.
\]
On $[0,1]^{\N}$ set
\[
 D(u,v)=\sum_{n=1}^\infty 2^{-n}|u_n-v_n|.
\]
The map $\Phi$ is injective. Indeed, if $x\neq y$ and $a=\widetilde\rho(x,y)>0$, choose $x_n$ with $\widetilde\rho(x,x_n)<a/3$. Then
$\widetilde\rho(y,x_n)>2a/3$, so the $n$th coordinates of $\Phi(x)$ and $\Phi(y)$ are different. Since $X$ is compact and the Hilbert cube is Hausdorff, $\Phi$ is a topological embedding. Define
\[
 \rho'(x,y)=D(\Phi(x),\Phi(y)).
\]
The reverse triangle inequality gives
\[
 \rho'(x,y)=\sum_{n=1}^\infty 2^{-n}|\widetilde\rho(x,x_n)-\widetilde\rho(y,x_n)|
 \leq \sum_{n=1}^\infty 2^{-n}\widetilde\rho(x,y)
 =\widetilde\rho(x,y).
\]
Thus $\rho'$ is a compatible metric and, in terms of the notation before normalization,
\[
 \rho'\leq\widetilde\rho\leq\rho.
\]

It remains to check tame growth. Given $0<\varepsilon<1/4$, choose $N$ so that $2^{-N}<\varepsilon/4$. Partition $[0,1]$ into at most $1+4/\varepsilon$ intervals of length at most $\varepsilon/4$, and use the resulting product partition on the first $N$ coordinates. If a cell meets $\Phi(X)$, choose one representative from that intersection. Two points of $\Phi(X)$ lying in the same cell have first-coordinate contribution to $D$ at most $\varepsilon/4$ and tail contribution at most $2^{-N}<\varepsilon/4$. In fact the distance between two points in the same cell is less than $\varepsilon/2$, so the chosen representatives are certainly $(D,\varepsilon)$-spanning set in $\Phi(X)$. Hence
\[
 \mathrm{S}_\varepsilon(X,\rho')=\mathrm{S}_\varepsilon(\Phi(X),D)
 \leq \left(1+\frac4\varepsilon\right)^N.
\]
Since $N=O(\log(1/\varepsilon))$, it follows that
\[
 \log \mathrm{S}_\varepsilon(X,\rho')
 =O\bigl((\log(1/\varepsilon))^2\bigr),
\]
which implies (3.4).
\end{proof}

The following estimate will play a key role in the subsequent calculations.

\begin{lemma}\label{lem:sparse}
Let $(X,\rho)$ be a compact metric space, $m \in \mathbb{N}$ and $A\subseteq X^m$. Let  $0<\varepsilon<1$ and $0<\alpha<1$. Put
\[
 \eta=\varepsilon^\alpha,
 \qquad
 \theta=\varepsilon^{1-\alpha}.
\]
Then
\[
 \mathrm{S}_{2\eta}(A,\rho_\infty)
 \leq
 \mathrm{S}_\varepsilon(A,\rho_1)
 \sum_{j=0}^{\lfloor\theta m\rfloor}
 \binom mj \mathrm{S}_\eta(X,\rho)^j.
 \tag{3.5}
\]
\end{lemma}

\begin{proof}
If $A=\varnothing$, then both sides of (3.5) are interpreted with the convention $\mathrm{S}_\varepsilon(\varnothing,\cdot)=0$, and the assertion is immediate. Assume $A\neq\varnothing$. Let $E\subseteq A$ be a $(\rho_1,\varepsilon)$-spanning set  and  $Q\subseteq X$  a $(\rho,\eta)$-spanning set. For every $\varphi\in A$, make and fix one choice
$\psi(\varphi)\in E$ such that
$\rho_1(\varphi,\psi(\varphi))<\varepsilon$, and define
\[
 J(\varphi)
 =\{v\in[m]:\rho(\varphi(v),\psi(\varphi)(v))\geq\eta\}.
\]
By Markov's inequality,
\[
 |J(\varphi)|\eta
 \leq \sum_{v=1}^m\rho(\varphi(v),\psi(\varphi)(v))
 <m\varepsilon,
\]
so $|J(\varphi)|<\theta m$.

For each $\varphi\in A$ and each $v\in J(\varphi)$, make and fix one choice
$q_v(\varphi)\in Q$ satisfying
$\rho(\varphi(v),q_v(\varphi))<\eta$. Define the code of $\varphi$ to be
\[
 \mathcal C(\varphi)
 =\left(\psi(\varphi),J(\varphi),
 (q_v(\varphi))_{v\in J(\varphi)}\right).
\]
Classify points of $A$ according to this code.
If $\varphi$ and $\varphi'$ have the same code, write the common code as
$(\psi,J,(q_v)_{v\in J})$. For $v\notin J$, we know that $\rho(\varphi(v),\psi(v))< \eta,\rho(\varphi'(v),\psi(v))< \eta$, then 
$$\rho(\varphi(v),\varphi'(v))\leq\rho(\varphi(v),\psi(v))+\rho(\varphi'(v),\psi(v))<2\eta.$$
For $v\in J$, we have $\rho(\varphi(v),q_v)<\eta,\rho(\varphi'(v),q_v)<\eta$, then
 $$\rho(\varphi(v),\varphi'(v))\leq\rho(\varphi(v),q_v)+\rho(\varphi'(v),q_v)<2\eta.$$Hence
\[
 \rho_\infty(\varphi,\varphi')<2\eta.
\]
Choose one representative from each nonempty class. These representatives form a $(\rho_{\infty},2\eta)$-spanning subset of $A$, and then $$\mathrm{S}_{2\eta}(A,\rho_\infty)\leq|\mathcal C(\varphi)|\leq \mathrm{S}_\varepsilon(A,\rho_1)
 \sum_{j=0}^{\lfloor\theta m\rfloor}
 \binom mj \mathrm{S}_\eta(X,\rho)^j.$$
 This completes the proof.
\end{proof}

Let
\[
 \cH(t)=-t\log t-(1-t)\log(1-t),\qquad 0<t<1,
\]
be the binary entropy function.

\begin{proposition}\label{prop:tame-equality}
Let $\Gamma$ be a countable discrete sofic group acting by homeomorphisms on a compact metrizable
space $X$ and $\Sigma$ a sofic approximation of $\Gamma$.
Let $\rho$ be a  compatible metric on $X$ with tame growth of covering numbers. Then, for every $1\leq p<\infty$,
\[
 \underline{\mdim}_{\Sigma,1}(X,\rho)
 =\underline{\mdim}_{\Sigma,p}(X,\rho)
 =\underline{\mdim}_{\Sigma,\infty}(X,\rho).
 \tag{3.6}
\]
The analogous equality also holds for the upper sofic $p$-metric mean dimension with respect to  $\rho$.
\end{proposition}

\begin{proof}
If the empty-model alternative in Remark~\ref{rem:empty-models} occurs, then every quantity in (3.6) is $-\infty$ and there is nothing to prove. We henceforth assume the other alternative; all scale covering rates below are then nonnegative.

Since
\[
 \rho_1\leq\rho_p\leq\rho_\infty,
\]
 we have
\[
 \underline{\mdim}_{\Sigma,1}(X,\rho)
 \leq\underline{\mdim}_{\Sigma,p}(X,\rho)
 \leq\underline{\mdim}_{\Sigma,\infty}(X,\rho).
 \tag{3.7}
\]
It suffices to prove the reverse inequality between the two endpoints.

Fix $0<\alpha<1$. Let $\eta=\varepsilon^\alpha$ and $\theta=\varepsilon^{1-\alpha}$. For $\varepsilon$ sufficiently small, $0<\theta<1/2$. Applying Lemma~\ref{lem:sparse} to
\[
 A=\Map(\rho,F,\delta,\sigma_i)
\]
and using
\[
 \binom mj\leq \exp\bigl(m\cH(j/m)\bigr)
 \leq \exp\bigl(m\cH(\theta)\bigr)
 \qquad(0\leq j\leq\theta m),
\]
we first note the explicit bound
\[
 \sum_{j=0}^{\lfloor\theta d_i\rfloor}
 \binom{d_i}{j}\mathrm{S}_\eta(X,\rho)^j
 \leq (d_i+1)\exp\bigl(d_i\cH(\theta)\bigr)
 \mathrm{S}_\eta(X,\rho)^{\theta d_i}.
\]
Consequently,
\begin{align*}
 \frac1{d_i}\log \mathrm{S}_{2\eta}(A,\rho_\infty)
 &\leq \frac1{d_i}\log \mathrm{S}_\varepsilon(A,\rho_1)
 +\cH(\theta)
 +\theta\log \mathrm{S}_\eta(X,\rho)
 +\frac{\log(d_i+1)}{d_i}.
\end{align*}
Taking $\limsup_i$ and then the infimum over $F$ and $\delta$ yields
\[
 h_{\Sigma,\infty}(2\varepsilon^\alpha,\rho)
 \leq h_{\Sigma,1}(\varepsilon,\rho)
 +\cH(\varepsilon^{1-\alpha})
 +\varepsilon^{1-\alpha}
   \log S_{\varepsilon^\alpha}(X,\rho).
 \tag{3.8}
\]
When $\epsilon\to 0$, tame growth gives
\[
 \varepsilon^{1-\alpha}
 \log \mathrm{S}_{\varepsilon^\alpha}(X,\rho)
 = (\varepsilon^\alpha)^{(1-\alpha)/\alpha}
 \log \mathrm{S}_{\varepsilon^\alpha}(X,\rho)
 \longrightarrow0.
 \tag{3.9}
\]
By (3.9) and the fact that $\cH(\varepsilon^{1-\alpha})\to0$, the total error term on the right-hand side of (3.8) tends to zero. Lemma~\ref{lem:scale-change}, applied with $c=2$, gives
\[
 L_\infty
 \leq \frac1\alpha L_1,
 \qquad
 L_1:=\underline{\mdim}_{\Sigma,1}(X,\rho),
 \quad
 L_\infty:=\underline{\mdim}_{\Sigma,\infty}(X,\rho).
 \tag{3.10}
\]
If $L_1=+\infty$, then (3.7) already implies $L_\infty=+\infty$. If $L_1<+\infty$, letting $\alpha\uparrow1$ in (3.10) gives $L_\infty\leq L_1$. Together with (3.7), this proves equality of the lower endpoint quantities and hence (3.6).

Applying the $\limsup$ part of Lemma~\ref{lem:scale-change} to the same inequality gives
\[
 \overline{\mdim}_{\Sigma,\infty}(X,\rho)
 \leq \frac1\alpha
 \overline{\mdim}_{\Sigma,1}(X,\rho).
\]
Letting $\alpha\uparrow1$, with the infinite case again settled by the elementary reverse inequality, proves the upper version.
\end{proof}

Moreover, the sparse-coordinate framework provides an estimate for the endpoint deviation even when tame growth fails to hold.

\begin{definition}[Covering-growth exponent]\label{def:growth-exponent}
For a compact metric space $(X,\rho)$ define
\[
 \beta(X,\rho)
 =\inf\left\{b\geq0:
 \lim_{\varepsilon\downarrow0}
 \varepsilon^a\log \mathrm{S}_\varepsilon(X,\rho)=0
 \text{ for every }a>b\right\},
 \tag{3.11}
\]
with the convention that the infimum of the empty set is $+\infty$.
\end{definition}

\begin{lemma}[Basic properties of the growth exponent]
\label{lem:beta-properties}
Let $(X,\rho)$ be a compact metric space.
\begin{enumerate}[label=\textup{(\roman*)}]
 \item If $\beta(X,\rho)<+\infty$, then for every $a>\beta(X,\rho)$,
 \[
  \lim_{\varepsilon\downarrow0}
  \varepsilon^a\log \mathrm{S}_\varepsilon(X,\rho)=0.
 \]
 \item One has $\beta(X,\rho)=0$ if and only if $(X,\rho)$ has tame growth of covering numbers.
\end{enumerate}
\end{lemma}

\begin{proof}
Let $\mathcal B$ be the set appearing on the right-hand side of (3.11). It is upward closed: if $b\in\mathcal B$ and $b'>b$, then $b'\in\mathcal B$. If $\beta(X,\rho)<a$, the definition of the infimum gives some $b\in\mathcal B$ with $b<a$. Since the defining limit for $b\in\mathcal B$ holds for every exponent strictly larger than $b$, it holds for $a$. This proves (i). If the metric is tame, then $0\in\mathcal B$, so $\beta(X,\rho)=0$. Conversely, if $\beta(X,\rho)=0$, part (i) gives the tame-growth limit for every $a>0$.
\end{proof}

\begin{proposition}
\label{prop:quantitative-comparison}
Let $\Gamma$ be a countable discrete sofic group acting by homeomorphisms on a compact metrizable
space $X$ and $\Sigma$ a sofic approximation of $\Gamma$.
Let $\rho$ be a compatible metric on $X$ and
assume $\beta(X,\rho)<+\infty$. Then, for every $1\leq p<\infty$,
\begin{align*}
  \underline{\mdim}_{\Sigma,1}(X,\rho)
 &\leq \underline{\mdim}_{\Sigma,p}(X,\rho)\\
 &\leq \underline{\mdim}_{\Sigma,\infty}(X,\rho) \leq
 \bigl(1+\beta(X,\rho)\bigr)
 \underline{\mdim}_{\Sigma,1}(X,\rho).
 \tag{3.12}   
\end{align*}

The analogous equality also holds for the upper sofic $p$-metric mean dimension with respect to  $\rho$.
\end{proposition}

\begin{proof}
If the empty-model alternative of Remark~\ref{rem:empty-models} occurs,
then all the quantities are $-\infty$, and the assertion is understood
as the equality of these quantities. Hence assume the nonempty-model
alternative. Then all scale covering rates and all fixed-metric
quantities are nonnegative.
The first two inequalities follow from
$\rho_1\leq\rho_p\leq\rho_\infty$. Fix any
$a>\beta(X,\rho)$ and put
\[
 \alpha=\frac1{1+a}.
\]
Lemma~\ref{lem:beta-properties} gives
\[
 \lim_{\eta\downarrow0}\eta^a\log \mathrm{S}_\eta(X,\rho)=0.
 \tag{3.13}
\]
Apply (3.8) with this value of $\alpha$. Since
$(1-\alpha)/\alpha=a$, the covering error in (3.8) equals
\[
 \varepsilon^{1-\alpha}
 \log \mathrm{S}_{\varepsilon^\alpha}(X,\rho)
 =\eta^a\log \mathrm{S}_\eta(X,\rho),
 \qquad \eta=\varepsilon^\alpha,
\]
and hence tends to zero by (3.13); the entropy error tends to zero as well. Lemma~\ref{lem:scale-change} therefore gives
\[
 \underline{\mdim}_{\Sigma,\infty}(X,\rho)
 \leq\frac1\alpha\underline{\mdim}_{\Sigma,1}(X,\rho)
 =(1+a)\underline{\mdim}_{\Sigma,1}(X,\rho).
\]
Letting $a\downarrow\beta(X,\rho)$ proves the last inequality in (3.12), with the infinite case automatic. The $\limsup$ part of Lemma~\ref{lem:scale-change} proves the upper version in exactly the same way.
\end{proof}

\begin{remark}
Proposition~\ref{prop:quantitative-comparison} is meaningful beyond the
tame case. For example, if
$\log \mathrm{S}_\varepsilon(X,\rho)=O(\varepsilon^{-b})$, then
$\beta(X,\rho)\leq b$ and the argument gives the endpoint factor
$1+b$. No optimality of this factor is claimed. When $\beta(X,\rho)=0$, the proposition recovers
Proposition~\ref{prop:tame-equality}.
\end{remark}

\begin{example}[Metrics with prescribed covering-growth exponent]
\label{ex:prescribed-beta}
Fix $b>0$ and let $X=\{0,1\}^{\N}$. For distinct $x,y\in X$, let
\[
 n(x,y)=\min\{k\geq1:x_k\neq y_k\},
 \qquad
 \rho_b(x,y)=n(x,y)^{-1/b},
\]
and set $\rho_b(x,x)=0$. Then $\rho_b$ is a compatible metric. If
\[
 (N+1)^{-1/b}<\varepsilon\leq N^{-1/b},
\]
the open $\varepsilon$-balls are exactly the cylinder sets determined by the first $N$ coordinates. Hence
\[
 S_\varepsilon(X,\rho_b)=2^N
 \quad\text{and therefore}\quad
 \log S_\varepsilon(X,\rho_b)\asymp\varepsilon^{-b}.
\]
It follows directly from Definition~\ref{def:growth-exponent} that
$\beta(X,\rho_b)=b$. Hence every positive exponent occurs, and Proposition~\ref{prop:quantitative-comparison} holds for a non-trivial class of non-tame metrics.
\end{example}

\section{Proof of Theorem~\ref{thm:strengthened}}\label{sec:main}
\begin{proof}[Proof of Theorem~\ref{thm:strengthened}]
If the global empty-model alternative of Remark~\ref{rem:empty-models} holds, then every fixed-metric and invariant-level quantity in the statement equals $-\infty$, and both (1.1) and (1.2) are immediate. We therefore assume the nonempty-model alternative. All quantities below then belong to $[0,+\infty]$.

Fix a $1\leq p < \infty$ and write
\[
 D_p=\mdim_{\Sigma,\mathrm M,p}(X,\Gamma),
 \qquad
 D_\infty=\mdim_{\Sigma,\mathrm M,\infty}(X,\Gamma).
\]
For every dynamically generating continuous pseudometric $\rho$, the pointwise inequality $\rho_p\leq\rho_\infty$ implies
\[
 \underline{\mdim}_{\Sigma,p}(X,\rho)
 \leq
 \underline{\mdim}_{\Sigma,\infty}(X,\rho).
\]
Taking infima gives
\[
 D_p\leq D_\infty.
 \tag{4.1}
\]
For finite $p$, Remark \ref{rem:compatible} allows $D_p$ to be computed by taking the infimum over compatible metrics.

Suppose  $D_p<+\infty$. Given $\kappa>0$, choose a compatible metric $\rho$ such that
\[
 \underline{\mdim}_{\Sigma,p}(X,\rho)<D_p+\kappa.
 \tag{4.2}
\]
By Lemma~\ref{lem:tame-minorant}, there exists $\rho'\in\mathcal T(X)$ with $\rho'\leq\rho$. Proposition~\ref{prop:monotone} and Proposition~\ref{prop:tame-equality} give
\begin{align*}
 D_\infty
 &\leq \underline{\mdim}_{\Sigma,\infty}(X,\rho')\\
 &=\underline{\mdim}_{\Sigma,p}(X,\rho')\\
 &\leq\underline{\mdim}_{\Sigma,p}(X,\rho)\\
 &<D_p+\kappa.
\end{align*}
Letting $\kappa\downarrow0$ yields $D_\infty\leq D_p$. If $D_p=+\infty$, (4.1) forces $D_\infty=+\infty$. Hence $D_p=D_\infty$ for every finite $p$. Denote this common value by $D_\Sigma(X,\Gamma)$; this proves (1.1).

It remains to prove (1.2). Let $1\leq q<\infty$. Since $\mathcal T(X)$ is a subclass of the compatible metrics, one has
\[
 D_\Sigma(X,\Gamma)=D_q
 \leq
 \inf_{\rho\in\mathcal T(X)}
 \underline{\mdim}_{\Sigma,q}(X,\rho).
 \tag{4.3}
\]

If $D_q<+\infty$, for each $\kappa>0$, we can find a compatible metric $\rho$ satisfying $\underline{\mdim}_{\Sigma,q}(X,\rho)<D_q+\kappa$. By Lemma \ref{lem:tame-minorant}, there exists $\rho'\in\mathcal T(X)$ with $\rho'\leq\rho$. Monotonicity gives
\[
 \inf_{\tau\in\mathcal T(X)}
 \underline{\mdim}_{\Sigma,q}(X,\tau)
 \leq
 \underline{\mdim}_{\Sigma,q}(X,\rho')
 \leq
 \underline{\mdim}_{\Sigma,q}(X,\rho)
 <D_q+\kappa.
\]
Letting $\kappa\downarrow0$, we have
 \[
 D_\Sigma(X,\Gamma)=D_q
 =
 \inf_{\rho\in\mathcal T(X)}
 \underline{\mdim}_{\Sigma,q}(X,\rho).
\]

If $D_q=+\infty$, (4.3) already forces the tame-metric infimum to be $+\infty$. Therefore
\[
 D_\Sigma(X,\Gamma)
 =\inf_{\rho\in\mathcal T(X)}
 \underline{\mdim}_{\Sigma,q}(X,\rho)
 \qquad(1\leq q<\infty).
 \tag{4.4}
\]
Finally, Proposition~\ref{prop:tame-equality} gives, for every $\rho\in\mathcal T(X)$,
\[
 \underline{\mdim}_{\Sigma,q}(X,\rho)
 =\underline{\mdim}_{\Sigma,\infty}(X,\rho).
\]
Taking the infimum over $\mathcal T(X)$ in (4.4) proves (1.2) also for $q=\infty$.
\end{proof}

\begin{corollary}[Answer to Hayes's question]\label{thm:main}
Under the hypotheses of Theorem~\ref{thm:strengthened}, for every $1\leq p<\infty$,
\[
 \mdim_{\Sigma,\mathrm M,p}(X,\Gamma)
 =\mdim_{\Sigma,\mathrm M,\infty}(X,\Gamma).
\]
\end{corollary}

\begin{remark}[Scope of the result]\label{rem:fixed}
Unlike a pure invariant-level equality, Theorem~\ref{thm:strengthened} 
provides one explicit regular metric class that simultaneously resolves all 
exponents.  It does not, however, imply the fixed-metric equality
\[
 \underline{\mdim}_{\Sigma,p}(X,\rho)
 =\underline{\mdim}_{\Sigma,\infty}(X,\rho)
\]
for every compatible metric $\rho$.  Hayes proved such an equality under a finite 
upper covering-exponent hypothesis (called finite packing dimension in 
\cite[Proposition~2.20]{Hayes}).  Proposition~\ref{prop:tame-equality} 
extends this to tame covering growth, permitting subpolynomial but 
super-logarithmic covering---the natural regularity class for the 
invariant-level theory.
\end{remark}

\begin{remark}[Why the metric minorant is essential]
We apply Lemma \ref{lem:hayes-input} only to establish monotonicity 
for $\rho'\le\rho$.  The constant $2$ in Proposition~\ref{prop:monotone} 
arises because $\varepsilon$-centers in a larger model space may lie 
outside the smaller one.  In the sparse-coordinate argument, the exceptional 
set and net points are fixed via the coding; thus two maps with identical 
codes differ by at most $2\eta$ coordinate-wise.  Lastly, the scaling 
$\varepsilon\mapsto2\varepsilon^\alpha$ introduces a $1/\alpha$ factor, 
eliminated only as $\alpha\to1^{-}$.  These observations clarify why any 
direct fixed-metric argument breaks down without tame growth.
\end{remark}

\section{Hilbert-cube metrics variational formula}
\label{sec:hilbert}

Let $\mathcal H(X)$ be the collection of metrics on $X$ of the form
\[
 \rho_{\mathbf f}(x,y)
 =\sum_{n=1}^{\infty}2^{-n}|f_n(x)-f_n(y)|,
 \tag{5.1}
\]
where $f_n\in C(X,[0,1])$ and the family $\mathbf{f}=(f_n)_{n\geq1}$ separates
points of $X$. The map
$x\mapsto(f_n(x))_{n\geq1}$ is then a topological embedding into the
Hilbert cube, so every member of $\mathcal H(X)$ is compatible.

\begin{lemma}
\label{lem:hilbert-cofinal}
Every metric in $\mathcal H(X)$ has tame growth of covering numbers.
Moreover, for every compatible metric $\rho$ on $X$ there exists
$\tau\in\mathcal H(X)$ such that $\tau\leq\rho$.
\end{lemma}

\begin{proof}
For the first assertion, the proof of Lemma~\ref{lem:tame-minorant}
applies directly to every subset of the Hilbert cube equipped with the
metric
\[
 D(u,v)=\sum_{n=1}^{\infty}2^{-n}|u_n-v_n|.
\]
It gives
\[
 \log \mathrm{S}_\varepsilon(X,\rho_{\mathbf f})
 =O\bigl((\log(1/\varepsilon))^2\bigr).
\]

For the second assertion, let $(x_n)_{n\geq1}$ be dense in $X$ and set
\[
 M=\max\{1,\diam(X,\rho)\},
 \qquad f_n(x)=\frac{\rho(x,x_n)}{M}.
\]
Then $f_n\in C(X,[0,1])$. The family $(f_n)$ separates points by the
same argument as in Lemma~\ref{lem:tame-minorant}. Hence the metric
\[
 \tau(x,y)=\sum_{n=1}^{\infty}2^{-n}|f_n(x)-f_n(y)|
\]
belongs to $\mathcal H(X)$. By the reverse triangle inequality,
\[
 \tau(x,y)\leq\frac{\rho(x,y)}{M}\leq\rho(x,y),
\]
as required.
\end{proof}

\begin{theorem}[Hilbert-cube variational formula]
\label{thm:hilbert-formula}
Under the hypotheses of Theorem~\ref{thm:strengthened}, for every
$1\leq q\leq\infty$,
\[
 D_\Sigma(X,\Gamma)
 =\inf_{\rho\in\mathcal H(X)}
 \underline{\mdim}_{\Sigma,q}(X,\rho).
 \tag{5.2}
\]
Consequently, a single explicit class of action-independent compatible metrics suffices to compute all exponents simultaneously.
\end{theorem}

\begin{proof}
If the empty-model alternative of Remark~\ref{rem:empty-models} holds, every term in (5.2) is $-\infty$, so the formula is immediate. Assume the nonempty-model alternative. Since $\mathcal H(X)\subseteq\mathcal T(X)$, Theorem~\ref{thm:strengthened} gives
\[
 D_\Sigma(X,\Gamma)
 \leq\inf_{\rho\in\mathcal H(X)}
 \underline{\mdim}_{\Sigma,q}(X,\rho).
 \tag{5.3}
\]
For the reverse inequality, first let $1\leq q<\infty$ and suppose $D_\Sigma(X,\Gamma)<+\infty$. Given a $\kappa>0$, by Remark \ref{rem:compatible}, we can find a compatible metric $\rho$ such that
\[
 \underline{\mdim}_{\Sigma,q}(X,\rho)
 <D_\Sigma(X,\Gamma)+\kappa.
\]
Choose $\tau\in\mathcal H(X)$ with $\tau\leq\rho$ by Lemma~\ref{lem:hilbert-cofinal}. Proposition~\ref{prop:monotone} gives
\[
 \inf_{\omega\in\mathcal H(X)}
 \underline{\mdim}_{\Sigma,q}(X,\omega)
 \leq\underline{\mdim}_{\Sigma,q}(X,\tau)
 \leq\underline{\mdim}_{\Sigma,q}(X,\rho)
 <D_\Sigma(X,\Gamma)+\kappa.
\]
Letting $\kappa\downarrow0$ proves equality. If the common invariant is $+\infty$, (5.3) already forces the right-hand side to be $+\infty$. This proves (5.2) for every finite $q$. Finally, every metric in $\mathcal H(X)$ is tame, so Proposition~\ref{prop:tame-equality} identifies its finite-$q$ and $\infty$ lower sofic metric mean dimensions with respect to  metric. Taking infima proves (5.2) for $q=\infty$.
\end{proof}

\begin{corollary}
\label{cor:simultaneous-near-minimizers}
Suppose $D_\Sigma(X,\Gamma)\in[0,+\infty)$. For every $\kappa>0$ there
exists a single metric $\rho_\kappa\in\mathcal H(X)$ such that,
simultaneously for every $1\leq q\leq\infty$,
\[
 D_\Sigma(X,\Gamma)
 \leq\underline{\mdim}_{\Sigma,q}(X,\rho_\kappa)
 <D_\Sigma(X,\Gamma)+\kappa.
\]
If $D_\Sigma(X,\Gamma)=-\infty$, there exists
$\rho_*\in\mathcal H(X)$ for which all these lower sofic $p$-metric mean dimensions with respect to $\rho_{*}$
are $-\infty$.
\end{corollary}

\begin{example}[A concrete Hilbert-cube metric on a full shift]
\label{ex:full-shift}
Let $X=[0,1]^\Gamma$ with the left shift action, and enumerate
$\Gamma=\{g_1,g_2,\ldots\}$. Define
\[
 \rho_{\mathrm{sh}}(x,y)
 =\sum_{n=1}^{\infty}2^{-n}|x(g_n)-y(g_n)|.
\]
The coordinate functions $x\mapsto x(g_n)$ separate points, so
$\rho_{\mathrm{sh}}$ is compatible with the product topology and belongs
to $\mathcal H(X)$. It therefore has tame growth of covering numbers, and
Proposition~\ref{prop:tame-equality} gives, for every $1\leq p<\infty$,
\[
 \underline{\mdim}_{\Sigma,1}(X,\rho_{\mathrm{sh}})
 =\underline{\mdim}_{\Sigma,p}(X,\rho_{\mathrm{sh}})
 =\underline{\mdim}_{\Sigma,\infty}(X,\rho_{\mathrm{sh}}),
\]
with the analogous equality for the upper fixed-metric quantities. Consequently, even for the canonical full shift, a single explicit metric in the product topology works uniformly across all exponents, so the fixed-metric comparison requires no metric tailored to $p$.
\end{example}

\enlargethispage{4\baselineskip}

\end{document}